\documentclass[11pt]{amsart}
\usepackage{amsfonts,amsmath,amsthm,amscd,amssymb,latexsym}
\usepackage{youngtab}
\usepackage{graphicx}
\usepackage{caption}
\numberwithin{equation}{subsection}
\newtheorem{theorem}{Theorem}[section]
\newtheorem{corollary}[theorem]{Corollary}
\newtheorem{definition}[theorem]{Definition}
\newtheorem{example}[theorem]{Example}
\newtheorem{lemma}[theorem]{Lemma}
\newtheorem{proposition}[theorem]{Proposition}
\newtheorem{remark}[theorem]{Remark}

\begin{document}

\title[The Littlewood-Richardson rule]{The Littlewood-Richardson Rule \\
and Gelfand-Tsetlin Patterns}

\author {Patrick Doolan}
\address{School of Mathematics and Physics, The University of 
Queensland, St Lucia, QLD 4072, Australia} 
\email{patrick.doolan@uqconnect.edu.au}

\author {Sangjib Kim}
\address{Department of Mathematics, Ewha Womans University,
Seoul, 120-750, South Korea} 
\email{sk23@ewha.ac.kr}

\begin{abstract}
We give a survey on the Littlewood-Richardson rule.
Using Gelfand-Tsetlin patterns as the main machinery of our analysis, 
we study the interrelationship of various combinatorial descriptions
of the Littlewood-Richardson rule.
\end{abstract}


\maketitle

\makeatletter
\def\Ddots{\mathinner{\mkern1mu\raise\p@
\vbox{\kern7\p@\hbox{.}}\mkern2mu
\raise4\p@\hbox{.}\mkern2mu\raise7\p@\hbox{.}\mkern1mu}}
\makeatother


\section{Introduction}


\subsection{}

Let us consider Schur polynomials $s_{\mu}$, $s_{\nu}$ and $s_{\lambda}$ in $n$ variables 
labelled by partitions $\mu, \nu$ and $\lambda$, respectively. 
\textit{The Littlewood-Richardson (LR) coefficient} is the multiplicity 
$c^{\lambda}_{\mu, \nu}$ of $s_{\nu}$ in the product of $s_{\mu}$ and $s_{\nu}$:
\begin{equation*}
s_{\mu} s_{\nu} = \sum_{\lambda} c^{\lambda}_{\mu, \nu} s_{\lambda}
\end{equation*}
and its description is called \textit{the LR rule}.

The same number appears in the tensor product decomposition problem in the representation 
theory of the complex general linear group $GL_n$ and Schubert calculus in the cohomology of the 
Grassmannians, and is also related to the eigenvalues of the sum of 
Hermitian matrices. For more details, we refer readers to \cite{Fu00, HL12, Ta04, vL01}.

\subsection{}

The LR rule is usually stated in terms of combinatorial objects 
called \textit{LR tableaux}. Recall that a Young tableau is a filling of the boxes
of a Young diagram with positive integers. We shall use the English convention 
of drawing Young diagrams and tableaux as in \cite{Fu97, St99} and assume a basic knowledge of these objects.

\begin{definition}
A tableau $T$ on a skew Young diagram is called a LR tableau if
it satisfies the following conditions:
\begin{enumerate}
\item it is semistandard, that is, the entries in each row of $T$
weakly increase from left to right, and the entries in each column
strictly increase from top to bottom; and

\item its reverse reading word is a Yamanouchi word (or lattice permutation). 
That is, in the word $x_1 x_2 x_3 \dots x_r$ obtained by reading all the entries 
of $T$ from left to right in each row starting from the bottom one,
the sequence $x_r x_{r-1} x_{r-2} \dots x_s$ contains at least as many $a$'s 
as it does $(a+1)$'s for all $a \geq 1$. 
\end{enumerate}
\end{definition}

For example, the following is a LR tableau on a skew Young 
diagram $(11,7,5,3)/(5,3,1)$
\begin{equation*}
\young(\ \ \ \ \ 111111,\ \ \ 1222,\ 2333,244)
\end{equation*}
and its reverse reading word is $24423331222111111$.

\begin{remark}\label{remark-Tx}
(1) In this paper we assume each tableau's entries  weakly 
increase from left to right in every row.
(2) From the second condition in the above definition, which we will call 
the \textit{Yamanouchi condition}, the $b$th row of a LR tableau does not 
contain any entries strictly bigger than $b$ for all $b\geq 1$.
\end{remark}

The number of LR tableaux on the skew shape $\lambda/\mu$ with 
content $\nu$ is equal to the LR number $c_{\mu,\nu}^{\lambda}$. Here,
the content $\nu=(\nu_1, \dotsc, \nu_n)$ of a tableau means that 
the entry $k$ appears $\nu_k$ times in the tableau for $k \geq 1$.
See, for example, \cite[\S I.9]{Ma95} and \cite{HL12}.

\subsection{}

In this paper, we survey variations of the semistandard and
Yamanouchi conditions with an emphasis on dualities in
combinatorial descriptions of the LR rule.
Although many of the results in this paper can be found in the literature, 
we will give complete and elementary proofs of our statements.

\smallskip

(1) In Theorem \ref{hive2gt2} and Theorem \ref{HT12}, we analyse \textit{hives}, 
introduced by Knutson and Tao along with their honeycomb model \cite{KT99}, 
in terms of Gelfand-Tsetlin(GT) patterns \cite{GT50}.
We then show how the interlacing conditions in GT patterns are intertwined to form the 
defining conditions of hives. For the relevant results, see for example \cite{BK, BZ1, BZ2, BZ3}.

(2) In Theorem \ref{LRGZ1}, we show that the semistandard and Yamanouchi
conditions in LR tableaux are equivalent to, respectively, the interlacing and 
exponent conditions in \textit{GZ schemes} introduced by Gelfand and Zelevinsky 
\cite{GZ85}. As a corollary we obtain a correspondence between LR tableaux and hives equivalent to \cite[(3.3)]{KTT1}. We then observe how conditions on LR tableaux, GZ schemes and hives are translated between objects by this bijection. For the relevant results, see, 
for example, \cite{BK, Bu00, KTT1, KTT2}.

(3) In Theorem \ref{CTGZ}, we show that the semistandard and Yamanouchi conditions 
in LR tableaux are equivalent to, respectively, the exponent and semistandard 
conditions in their \textit{companion tableaux} introduced by van Leeuwen \cite{vL01}. 
Here the correspondence between conditions is obtained by taking the transpose of matrices.

As a consequence, we obtain bijections between the families of combinatorial objects
counting the LR number.

\subsection{}

\smallskip

In \cite{HTW05, HJLTW09}, Howe and his collaborators constructed a polynomial model 
for the tensor product of representations in terms of two copies of 
the multi-homogeneous coordinate ring of the flag variety, and then studied its 
toric degeneration with the SAGBI-Gr\"{o}bner method. Through
the characterization of the leading monomials of highest weight vectors, their toric 
variety is encoded by the \textit{LR cone} \cite{PV05}.
On the other hand, via toric degenerations, the flag variety may be described 
in terms of the lattice cone of GT patterns \cite{GL96, Ki08, KM05}.
These results led us to study the LR rule in terms of two sets of
interlacing or semistandard conditions and to investigate the interrelationship 
of various combinatorial descriptions of the LR rule with GT patterns.

\smallskip


\section{Hives and GT Patterns I}


In this section, we define GT patterns, hives, and objects 
related to them. We also describe hives in terms of pairs of GT patterns.

\subsection{}

We set, once and for all, three polynomial dominant weights of the complex 
general linear group $GL_n$, that is, the sequences of nonnegative integers:
\begin{equation*}
\lambda = (\lambda_1, \dotsc, \lambda_n),\  \  
\mu = (\mu_1, \dotsc, \mu_n), \ \  \nu = (\nu_1, \dotsc, \nu_n)
\end{equation*}
such that $\lambda_i \geq \lambda_{i+1}$, $\mu_i \geq \mu_{i+1}$, 
and $\nu_i \geq \nu_{i+1}$ for all $i$.
We define the dual $\lambda^*$ of $\lambda$ to be 
\begin{equation*}
\lambda^*=(-\lambda_n, -\lambda_{n-1}, \dotsc, -\lambda_1),
\end{equation*} 
and define $\mu^*$ and $\nu^*$ similarly.

\subsection{}

Let us consider an array of integers, which we will call
a \textit{$t$-array}
$$ T = \left (t_1^{(1)}, \dots, t_j^{(i)}, \dots, 
t_n^{(n)} \right) \in \mathbb{Z}^{n(n+1)/2} $$
where $1 \leq j  \leq i \leq n$. 
We are particularly interested in the case when the entries of 
$T$ are either all non-negative or all non-positive integers.

\begin{definition}
A $t$-array $T=(t^{(i)}_j) \in \mathbb{Z}^{n(n+1)/2}$ is called a GT pattern 
for $GL_n$ if it satisfies the \textit{interlacing conditions}:
\begin{align*}
\hbox{IC(1):} &\  t_{j}^{(i+1)} \geq t_{j}^{(i)} \\
\hbox{IC(2):} &\  t_{j}^{(i)} \geq t_{j+1}^{(i+1)}
\end{align*}
for all $i$ and $j$.
\end{definition}

We shall draw a $t$-array in the reversed pyramid form. For example,
a generic GT pattern for $GL_5$ is%
\begin{equation*}
\begin{array}{ccccccccc}
t_{1}^{(5)} &  & t_{2}^{(5)} &  & t_{3}^{(5)} &  & t_{4}^{(5)} &  &
t_{5}^{(5)} \\
& t_{1}^{(4)} &  & t_{2}^{(4)} &  & t_{3}^{(4)} &  & t_{4}^{(4)} &  \\
&  & t_{1}^{(3)} &  & t_{2}^{(3)} &  & t_{3}^{(3)} &  &  \\
&  &  & t_{1}^{(2)} &  & t_{2}^{(2)} &  &  &  \\
&  &  &  & t_{1}^{(1)} &  &  &  &
\end{array}%
\end{equation*}
where the entries are weakly decreasing along the diagonals 
from left to right.

Then, the \textit{dual array} $T^{\ast}=(s_j^{(i)})$ of $T$ is the $t$-array 
obtained by reflecting $T$ over a vertical line and then multiplying $-1$, i.e.,
$$s_j^{(i)} = -t_{i+1-j}^{(i)}$$
for all $1 \leq j  \leq i \leq n$.

\begin{definition}\label{def-weights}
For a $t$-array $T=(t_{j}^{(i)}) \in \mathbb{Z}^{n(n+1)/2}$, 
\begin{enumerate}
\item the $k$th row of $T$ is 
$t^{(k)}=(t_{1}^{(k)},t_{2}^{(k)},\dotsc ,t_{k}^{(k)}) \in \mathbb{Z}^k$
for $1\leq k\leq n$. The type of $T$ is its $n$th row;

\item the weight of $T$ is $(w_{1},w_{2},\dotsc ,w_{n}) \in \mathbb{Z}^n$ where
$w_1 = t_{1}^{(1)}$ and
\begin{equation*}
w_i = \sum_{k=1}^{i} t_{k}^{(i)} - \sum_{k=1}^{i-1} t_{k}^{(i-1)} 
\hbox{\ \ \ for\ \ } 2\leq i\leq n.
\end{equation*}
\end{enumerate}
\end{definition}
Note that if $T$ is of type $\lambda$ and weight $w \in \mathbb{Z}^n$, then
$T^*$ is of type $\lambda^*$ and weight $-w$.

\smallskip

GT patterns were introduced by Gelfand and Tsetlin in \cite{GT50} to label
the weight basis elements of an irreducible representation of the general linear group. 
The weight of $T$ is exactly the weight of the basis element labelled by $T$ in the
irreducible representation $V_n^{\mu}$ whose highest weight is $\mu=t^{(n)}$. It follows 
that the dual array $T^*$ of $T$ corresponds to a weight vector in the contragradiant
representation of $V_n^{\mu}$.

\subsection{}

Let us consider an array of nonnegative integers, which we will 
call a \textit{$h$-array},
\begin{equation*}
\left(h_{0,0}, \dots, h_{a,b}, \dots, h_{n,n} \right) 
\in \mathbb{Z}^{(n+1)(n+2)/2}
\end{equation*}%
where $0\leq a\leq b\leq n$ and $h_{0,0}=0$. 

\begin{definition}
A hive for $GL_n$ is a $h$-array $H=(h_{a,b}) \in \mathbb{Z}^{(n+1)(n+2)/2}$
satisfying the \textit{rhombus conditions}:
\begin{align*}
\hbox{RC(1):} &\  (h_{a,b}+h_{a-1,b-1}) \geq (h_{a-1,b}+h_{a,b-1}) 
\hbox{\ \ for\ \ } 1 \leq a < b \leq n, \\
\hbox{RC(2):} &\  (h_{a-1,b}+h_{a,b}) \geq (h_{a,b+1}+h_{a-1,b-1}) 
\hbox{\ \ for\ \ } 1 \leq a \leq b < n,  \\
\hbox{RC(3):} &\  (h_{a,b}+h_{a,b+1}) \geq (h_{a+1,b+1}+h_{a-1,b}) 
\hbox{\ \ for\ \ } 1 \leq a \leq b <n.
\end{align*}%
\end{definition}

We shall draw a $h$-array in the pyramid form. For example, 
a generic hive for $GL_3$ is shown below.
\begin{equation*}
\begin{array}{ccccccccccccc}
&  &  &  &  &  & h_{0,0} &  &  &  &  &  &  \\
&  &  &  &  &  &  &  &  &  &  &  &  \\
&  &  &  & h_{0,1} &  &  &  & h_{1,1} &  &  &  &  \\
&  &  &  &  &  &  &  &  &  &  &  &  \\
&  & h_{0,2} &  &  &  & h_{1,2} &  &  &  & h_{2,2} &  &  \\
&  &  &  &  &  &  &  &  &  &  &  &  \\
h_{0,3} &  &  &  & h_{1,3} &  &  &  & h_{2,3} &  &  &  & h_{3,3}%
\end{array}%
\end{equation*}
The rhombus conditions RC(1), RC(2), and RC(3) then say that, 
for each fundamental rhombus of one of the following forms,
\begin{equation*}
\begin{array}{ccccccccccccccccccccc}
  &   &  &  &   &  & &  &   & & A&  &   & & & &  &  &  &  & \\
  &   &  &  &   &  & &  &   & & &  &   & & & &  &  &  &  & \\
  & O' &  &  & A' &  & &  &O & & &  & O' & & & &  A'&  &  &O & \\
  &   &  &  &   &  & &  &   & & &  &   & & & &  &  &  &  & \\
 A&   &  & O&   &, & &  &   & & A'&  &   &,& & &  & O'&  &  & A%
\end{array}%
\end{equation*}
the sum of entries at the obtuse corners is bigger than or equal 
to the sum of entries at the acute corners, i.e., $O + O' \geq A + A'$.

For polynomial dominant weights $\mu$, $\nu$, and $\lambda$ of $GL_n$, 
we let $\mathcal{H}(\mu,\nu,\lambda)$ denote the set of all $h$-arrays such that
\begin{align} \label{hivebdy}
\mu &= (h_{0,1}-h_{0,0}, h_{0,2}- h_{0,1},\dotsc, h_{0,n} - h_{0,n-1}), \notag \\
\nu &= (h_{1,n}-h_{0,n}, h_{2,n}- h_{1,n},\dotsc, h_{n,n} - h_{n-1,n}), \\
\lambda &= (h_{1,1}-h_{0,0}, h_{2,2}- h_{1,1},\dotsc, h_{n,n} - h_{n-1,n-1}). \notag
\end{align}
That is, the three boundary sides of $H \in \mathcal{H}(\mu,\nu,\lambda)$ are fixed:
\begin{align*}
h_{0,i} &= \mu_1 + \mu_2 + \dotsb + \mu_i \\ 
h_{i,n} &= \sum_{j=1}^n \mu_j + \nu_1 + \nu_2 + \dotsb + \nu_i \\
h_{i,i} &= \lambda_1 + \lambda_2 + \dotsb + \lambda_i
\end{align*}
for $1 \leq i \leq n$. Recall that we always set $h_{0,0}=0$. Let $\mathcal{H}^{\circ}(\mu,\nu,\lambda)$ be the subset of $\mathcal{H}(\mu,\nu,\lambda)$ satisfying the rhombus conditions. This is the set of hives whose boundaries are described by \eqref{hivebdy}.

Hives were introduced by Knutson and Tao in \cite{KT99} along with their honeycomb model 
to prove the saturation conjecture. In particular, the number of hives in 
$\mathcal{H}^{\circ}(\mu,\nu,\lambda)$ is equal to the LR number $c_{\mu,\nu}^{\lambda}$. 
See also \cite{Bu00, KTW04, PV05}.

\subsection{}

For each $h$-array $H=(h_{a,b})\in \mathbb{Z}^{(n+1)(n+2)/2}$, let us define 
its \textit{derived $t$-arrays} 
\begin{equation*}
T_1 = (x_j^{(i)}),\ \ \  T_2 = (y_j^{(i)}), \ \ \  T_3 = (z_j^{(i)})
\end{equation*}
whose entries are obtained from the differences of adjacent entries of $H$. 
\begin{figure}[ht!]
\begin{center}
\includegraphics[scale=0.5]{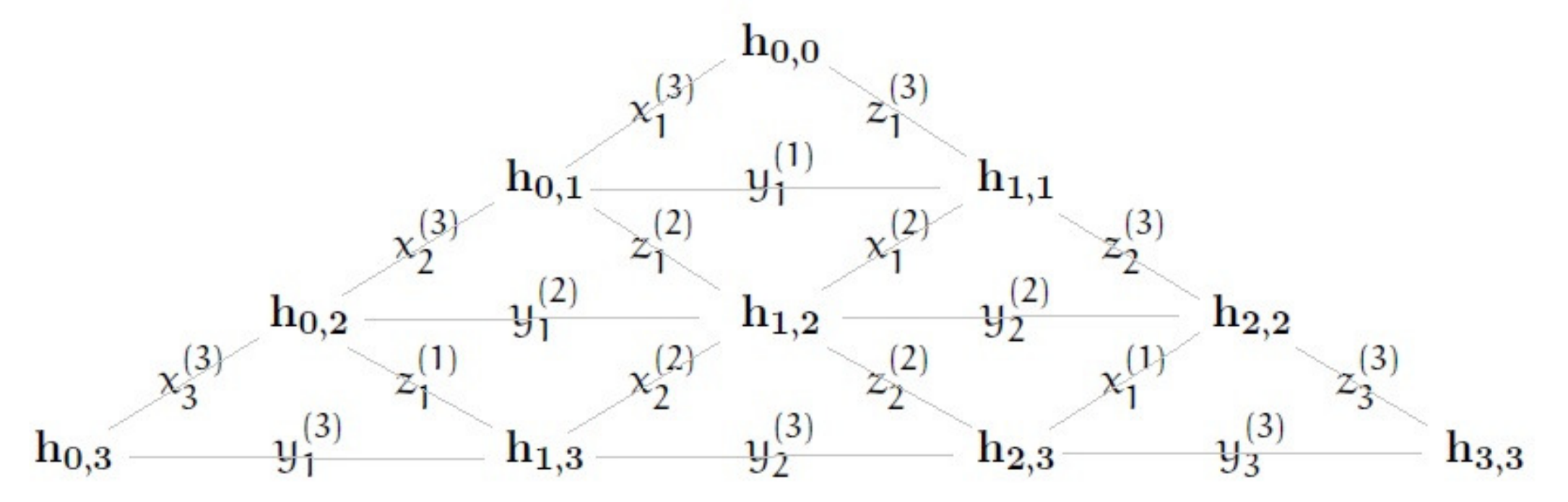}
\caption{A $h$-array and its three derived $t$-arrays.}
\end{center}
\end{figure}

More specifically, for each fundamental triangle in $H$,
\begin{equation*}
\begin{array}{ccccc}
&  & h_{a,b} &  &  \\
&  &  &  &  \\
h_{a,b+1} &  &  &  & h_{a+1,b+1}%
\end{array}
\end{equation*}
the entries of the derived $t$-arrays $(x_j^{(i)})$, $(y_j^{(i)})$, 
and $(z_j^{(i)})$ are 
\begin{align}\label{derived-arrays12}
x_{b+1-a}^{(n-a)} &= h_{a,b+1}-h_{a,b} \hbox{\ \ \ \ (SW--NE direction)}\notag \\
y_{a+1}^{(b+1)} &= h_{a+1,b+1}-h_{a,b+1} \hbox{\ \ (E--W direction)}  \\ 
z_{a+1}^{(n+a-b)} &= h_{a+1,b+1}-h_{a,b} \hbox{\ \ \ \ (SE--NW direction)} \notag
\end{align}%
for $0\leq a\leq b\leq n-1$. 

This rather involved indexing is to make
the entries of the derived arrays compatible with those of GT patterns.
We may visualize the derived $t$-arrays by placing their entries
between the entries of the $h$-array used to compute them.
For example, if $n=3$, then a $h$-array and its three derived $t$-arrays 
may be drawn as Figure 1.

\subsection{}

The rhombus conditions for $h$-arrays are closely related to 
the interlacing conditions for their derived $t$-arrays.
\begin{proposition}\label{IC-Rhom}
Let $T_k=T_k(H)$ be a derived $t$-array of a $h$-array $H$ for $k=1,2,3$.
\begin{enumerate}
\item $H$ satisfies {RC(1)} if and only if $T_1$ satisfies IC(2) and $T_2$ satisfies {IC(1)}.
\item $H$ satisfies {RC(2)} if and only if $T_1$ and $T_3$ satisfy {IC(1)}.
\item $H$ satisfies {RC(3)} if and only if $T_2$ and $T_3$ satisfy {IC(2)}.
\item $T_3$ satisfies {IC(1)} if and only if $T_1$ satisfies {IC(1)}.
\item $T_3$ satisfies {IC(2)} if and only if $T_2$ satisfies {IC(2)}.
\end{enumerate}
\begin{proof}
Let us consider five adjacent entries of $H$ of the forms
\begin{small}
\begin{equation*}
\begin{array}{cccccccccccccccccccccc}
  &   &  & Z_1&   &  &   & &  &  &  &  &  &  &    Z_3&  &   &  \\
  &   &  &  &   &  & &   &  &  &  &  &  &  & &      &   &   \\
  & Y_1 &  &  & W_1 &   &  & & Y_2&  &  &W_2 &    &  Y_3&  &  & W_3 &  \\
  &   &  &  &   &  & &   &  &  &  &  &  &  & &    &     &   \\
 X_1&   &  & V_1&   &, &  X_2& &  &V_2 &  &  &U_2,    & & V_3&  &   &U_3.%
\end{array}%
\end{equation*}
\end{small}
Then, in the first and the third ones, RC(2) says that $Y_i+W_i \geq Z_i +V_i$ 
for $i=1$ and $3$. This is equivalent to $Y_1 - Z_1 \geq V_1 - W_1$ and 
$W_3 - Z_3 \geq V_3 - Y_3$, which are IC(1) for $T_1$ and $T_3$, respectively. 
This proves the statement (2). The statements (1) and (3) can be shown similarly.

Next, let us consider fundamental rhombi of the following forms in $H$
\begin{small}
\begin{equation*}
\begin{array}{ccccccccccccccc}
   &  & K&   &   & & & & &  & &  &  &  &    \\
   &  &  &   &   & & & & &  & &  &  &  &   \\
 L &  &  &   & N & & & & &P & &  & S&  &    \\
   &  &  &   &   & & & & &  & &  &  &  &    \\
   &  & M&   &   &,& & & &  & & Q&  &  & R .%
\end{array}
\end{equation*}
\end{small}
Note that $N-K \geq M-L$ if and only if $L-K \geq M-N$, which proves (4).
Similarly, $P-Q \geq S-R$ if and only if $P-S \geq Q-R$, which proves (5).
\end{proof}
\end{proposition}

Suppose a $h$-array $H$ satisfies RC(1), RC(2), and RC(3). Then, by the statements (1) and (2) of Proposition \ref{IC-Rhom},
$T_1(H)$ satisfies IC(1) and IC(2). Similarly, by the statements (1) and (3), $T_2(H)$ satisfies IC(1) and IC(2). This shows
that $T_1(H)$ and $T_2(H)$ are GT patterns. Conversely, if $T_1(H)$ and $T_2(H)$ are GT patterns, then, 
by the statements (4) and (5), $T_3(H)$ is also a GT pattern. This means all three derived $t$-arrays satisfy both 
IC(1) and IC(2), and therefore, from the statements (1), (2), and (3), $H$ is a hive.

\begin{theorem}\label{hive2gt2}
For a $h$-array $H\in \mathbb{Z}^{(n+1)(n+2)/2}$ and its derived 
$t$-arrays $T_1(H)$ and $T_2(H)$, $H$ is a hive if and only if $T_1(H)$ 
and $T_2(H)$ are GT patterns for $GL_n$.
\end{theorem}

We remark that, in the above result, $T_1(H)$ and $T_2(H)$ are not independent.
Let $T_1=(x^{(i)}_j)$ and $T_2=(y^{(i)}_j)$ be the derived 
$t$-arrays of a $h$-array $H$. Then, for each rhombus of the form
\begin{equation*}
\begin{array}{cccccc}
 & & B & & & A \\ 
 & &   & & &   \\
 C& &  &D & &
\end{array}
\end{equation*} 
we have $(D-C)+(C-B)=(D-A)+(A-B)$, or 
\begin{equation*}
(C-B)-(D-A)=(A-B)-(D-C)
\end{equation*}
which is,
using \eqref{derived-arrays12},
\begin{equation}\label{tete}
x_{b-a}^{(n-a-1)} - x_{b+1-a}^{(n-a)} = y_{a+1}^{(b+1)} - y_{a+1}^{(b)}
\end{equation}
for $0 \leq a < b < n$.
Note that hives (respectively, GT patterns) for $GL_n$ with non-negative entries 
form a subsemigroup of $\mathbb{Z}^{(n+1)(n+2)/2}_{\geq 0}$ 
(respectively, $\mathbb{Z}^{n(n+1)/2}_{\geq 0}$). 
Theorem \ref{hive2gt2} and \eqref{tete} imply that the semigroup 
\begin{equation*}
\bigcup_{(\mu,\nu,\lambda)}\mathcal{H}^{\circ}(\mu,\nu,\lambda)
\end{equation*} 
of hives is a fiber product of, 
over $\mathbb{Z}_{\geq 0}^{n(n-1)/2}$, two affine semigroups $S_{GT}^1$ 
and $S_{GT}^2$ of GT patterns with respect to
\begin{equation*}
\phi_k : S_{GT}^k \longrightarrow \mathbb{Z}_{\geq 0}^{n(n-1)/2}
\end{equation*}
such that, for $0 \leq a < b <n$, 
\begin{align*}
\phi_1(T_1) &= \left( \dots , x_{b-a}^{(n-a-1)}  - x_{b+1-a}^{(n-a)}, \dots \right), \\ 
\phi_2(T_2) &= \left( \dots , y_{a+1}^{(b+1)} - y_{a+1}^{(b)}, \dots \right)
\end{align*}
where $T_1 = (x_j^{(i)}) \in S_{GT}^1$ and $T_2 = (y_j^{(i)}) \in S_{GT}^2$.

We also remark that by exchanging the roles of $T_1(H)$, $T_2(H)$ and $T_3(H)$,
one can read the symmetry of the LR rule. See, for example, \cite{TY}.


\section{Hives and GT Patterns II}


In this section, we study the set $\mathcal{H}^{\circ}(\mu, \nu, \lambda)$ of hives 
with a given boundary condition in terms of a single GT pattern.

\subsection{}

Gelfand and Zelevinsky counted the LR number $c_{\mu,\nu}^{\lambda}$ with 
GT patterns of type $\mu$ and weight $\lambda-\nu$ satisfying the following additional condition.

\begin{lemma}\cite{GZ85}
For a $t$-array $T=(t_{j}^{(i)}) \in \mathbb{Z}^{n(n+1)/2}$, we define its exponents as
\begin{equation*}
\varepsilon _{j}^{(i)}(T)=\sum_{1\leq
h<j}(t_{h}^{(i+1)}-2t_{h}^{(i)}+t_{h}^{(i-1)})+(t_{j}^{(i+1)}-t_{j}^{(i)}).
\end{equation*}%
Then the cardinality of the set $GZ(\mu,\lambda-\nu,\nu)$ of all GT
patterns $T$ of type $\mu$ with weight $\lambda-\nu$ such that, 
for all $i$ and $j$, 
\begin{equation*}
\varepsilon _{j}^{(i)}(T)\leq \nu_{i} - \nu_{i+1}
\end{equation*}%
is equal to the LR number $c_{\mu,\nu}^{\lambda}$.
\end{lemma}
The elements of $GZ(\mu,\lambda-\nu,\nu)$ will be called \textit{GZ schemes}.

\subsection{}

Note that, for a $h$-array $H$, since the derived $t$-arrays are defined from the differences 
of the entries in $H$, if the boundaries of $H$ are fixed, then 
any one of the derived $t$-array of $H$ uniquely determines $H$. Moreover, we can 
characterize the derived $t$-arrays as follows.

\begin{theorem}\label{HT12}
For a $h$-array $H$ in $\mathcal{H}(\mu,\nu,\lambda)$, consider its derived $t$-arrays
$T_1(H)$ and $T_2(H)$.
\begin{enumerate}
\item $H$ is a hive if and only if $T^*_1(H)=(T_1(H))^*$ is a GZ scheme 
in $GZ(\mu^*, \lambda^* - \nu^*,\nu^*)$;

\item $H$ is a hive if and only if $T_2(H)$ is a GZ scheme 
in $GZ(\nu, \lambda-\mu, \mu)$.
\end{enumerate}
\end{theorem}

Note that this theorem, in particular, gives bijections between hives and GZ schemes:
\begin{equation*}
\begin{array}{ccc}
\mathcal{H}^{\circ}(\mu,\nu,\lambda) & \longrightarrow         & GZ(\mu^*, \lambda^* - \nu^*,\nu^*) \\
 H                           & \longmapsto             & T^*_1(H)                                                            
\end{array}
\end{equation*}
and
\begin{equation*}
\begin{array}{cccc}
\mathcal{H}^{\circ}(\mu,\nu,\lambda) & \longrightarrow    & GZ(\nu, \lambda-\mu, \mu) & \\
 H                          & \longmapsto         & T_2(H)  & ^{.}
\end{array}
\end{equation*}

For the rest of this section, we will prove Theorem \ref{HT12} by showing the following.
\begin{enumerate}
\item[(a)] $T^*_1(H)$ satisfies IC(2) if and only if $\varepsilon^{(i)}_j(T_2(H)) \leq \mu_i - \mu_{i+1}$;

\item[(b)] $T^*_1(H)$ satisfies IC(1) if and only if $T_2(H)$ satisfies IC(1);

\item[(c)] $T^*_1(H)$ satisfies $\varepsilon^{(i)}_j(T_1^*(H)) \leq \nu^*_i - \nu^*_{i+1}$ 
if and only if $T_2(H)$ satisfies IC(2).
\end{enumerate}
The weights of the derived $t$ arrays will also be computed. 

\subsection{}

Let us first compute the weights of $T_1(H)$ and $T_2(H)$ 
for $H \in \mathcal{H}(\mu,\nu,\lambda)$.
\begin{lemma}\label{l-HT12-w}
For a $h$-array $H=(h_{a,b})\in \mathcal{H}(\mu ,\nu,\lambda )$,
\begin{enumerate}
\item the weight of $T_1(H)$ is $\nu^* - \lambda^*$, i.e., 
$$(\lambda_{n}- \nu_{n}, \lambda_{n-1} - \nu_{n-1}, \dotsc, \lambda_1 - \nu_1)$$
therefore, the weight of $T^*_1(H)$ is $\lambda^* - \nu^*$;

\item the weight of $T_2(H)$ is $\lambda - \mu$, i.e.,
$$(\lambda_1 - \mu_1, \lambda_2 - \mu_2, \dotsc, \lambda_n - \mu_n).$$
\end{enumerate}
\end{lemma}
\begin{proof}
We will prove the second statement. The proof of the first case is similar. From Definition \ref{def-weights}, \eqref{derived-arrays12} and the expressions for $\lambda$ and $\mu$ in terms of the $h$-array elements it follows
\begin{align*}
w_1  = y_{1}^{(1)}   = h_{1,1} - h_{0,1} = (h_{1,1} - h_{0,0}) + (h_{0,0}- h_{0,1}) = \lambda_{1} - \mu_{1} . 
\end{align*} 
Using the same approach for $w_i$, $i \geq 2$, we see 
\begin{align*}
w_i  &= \sum_{k=1}^{i} y_{k}^{(i)} - \sum_{k=1}^{i-1} y_{k}^{(i-1)} \\
       & = \sum_{k=1}^{i} (h_{k,i} - h_{k-1,i}) -  \sum_{k=1}^{i-1} (h_{k,i-1} - h_{k-1,i-1}) \\
       & = (h_{i,i} - h_{0,i}) - (h_{i-1,i-1} - h_{0,i-1}) \\
       & = \lambda_i  - \mu_i .
\end{align*} 
Therefore $w_i = \lambda_i - \mu_i$ for all $i$, and the weight of $T_2(H)$ is $\lambda - \mu$.
\end{proof}

\subsection{}

Next, we study the relations between the interlacing conditions and the exponents conditions
for derived arrays. Note that, from the definition of dual arrays, a $t$-array $T$ satisfies 
IC(1) if and only if $T^*$ satisfies IC(2), and $T$ satisfies IC(2) if and only if $T^*$ satisfies IC(1).

\begin{proposition}\label{p-HT12-1}
For a $h$-array $H=(h_{a,b}) \in \mathcal{H}(\mu,\nu,\lambda)$ and 
its derived $t$-arrays $T_1(H)=(x^{(i)}_j)$ and $T_2(H)=(y^{(i)}_j)$,
$T_1(H)$ satisfies IC(1) if and only if 
$\varepsilon^{(i)}_j(T_2(H)) \leq \mu_i - \mu_{i+1}$. 
\end{proposition}
\begin{proof}

Let us assume $j>1$. Then the exponent of $T_2(H)$,
\begin{equation*}
\varepsilon _{j}^{(i)}(T_2(H)) = \sum_{1\leq h < j}
\left( (y_{h}^{(i+1)}-y_{h}^{(i)}) - (y_{h}^{(i)} - y_{h}^{(i-1)}) \right)
+ \left( y_{j}^{(i+1)}-y_{j}^{(i)} \right)
\end{equation*}
can be rewritten in terms of the entries of $T_1(H)$. By using \eqref{tete},
\begin{eqnarray*}
\varepsilon _{j}^{(i)}(T_2(H)) &=& \sum_{1\leq h < j}
\left( (x_{i-h+1}^{(n-h)} - x_{i-h+2}^{(n-h+1)}) - (x_{i-h}^{(n-h)} - x_{i-h+1}^{(n-h+1)}) \right) \\
&& + \left( x_{i-j+1}^{(n-j)} - x_{i-j+2}^{(n-j+1)} + y_j^{(i)} \right)  
 - \left( x_{i-j}^{(n-j)}  - x_{i-j+1}^{(n-j+1)} + y_j^{(i-1)} \right)
\end{eqnarray*}
and we see that parts of the consecutive terms cancel to give
\begin{equation}\label{exp-t2h}
\varepsilon _{j}^{(i)}(T_2(H)) 
= \left( x_i^{(n)} - x_{i+1}^{(n)} \right) + \left( x_{i-j+1}^{(n-j)}
 - x_{i-j}^{(n-j)}  + y_j^{(i)} - y_j^{(i-1)}\right).
\end{equation}

Now note that the interlacing condition IC(1) for $T_1(H)$ implies 
$x_{i-j+1}^{(n-j+1)} \geq x_{i-j+1}^{(n-j)}$
or equivalently, by using \eqref{tete},
\begin{eqnarray*}
x_{i-j}^{(n-j)} &\geq & \left( x_{i-j+1}^{(n-j)} +  y_{j}^{(i)} - y_{j}^{(i-1)} \right) 
\end{eqnarray*}
therefore
\begin{eqnarray*}
0 &\geq & \left( x_{i-j+1}^{(n-j)} - x_{i-j}^{(n-j)} +  y_{j}^{(i)} - y_{j}^{(i-1)} \right).
\end{eqnarray*}
Hence, from \eqref{exp-t2h}, the interlacing condition 
IC(1) for $T_1(H)$ is equivalent to 
\begin{equation*}
\varepsilon _{j}^{(i)}(T_2(H)) \leq \left( x_i^{(n)} - x_{i+1}^{(n)} \right)
= \mu_i - \mu_{i+1}.
\end{equation*}

The case $j=1$ can be shown similarly for all $i$.
\end{proof}

\begin{proposition}\label{p-HT12-2}
For a $h$-array $H=(h_{a,b}) \in \mathcal{H}(\mu,\nu,\lambda)$ and 
its derived $t$-arrays $T_1(H)=(x^{(i)}_j)$ and $T_2(H)=(y^{(i)}_j)$,
$T_1(H)$ satisfies IC(2) if and only if $T_2(H)$ satisfies IC(1).
\end{proposition}
\begin{proof}
 
Using the equality \eqref{tete},
\begin{equation*}
\left( x_{j}^{(i)} \geq x_{j+1}^{(i+1)} \right) \hbox{\ \ \ if and only if \ \ } 
\left( y_{n-i}^{(n-i + j)} \geq y_{n-i}^{(n-i+j-1)} \right)
\end{equation*}
and therefore, by setting $i'=n-i+j-1$ and $j'=n-i$, we have
\begin{equation*}
\left( x_{j}^{(i)} \geq x_{j+1}^{(i+1)} \right) \hbox{\ \ \ if and only if \ \ } 
\left( y_{j'}^{(i'+1)} \geq y_{j'}^{(i')} \right)
\end{equation*}
for $1 \leq j \leq i \leq n-1$ and $1 \leq j' \leq i' \leq n-1$. This shows that
IC(2) holds for $T_1(H)$ if and only if IC(1) holds for $T_2(H)$. 
\end{proof}

\begin{proposition}\label{p-HT12-3}
For a $h$-array $H=(h_{a,b}) \in \mathcal{H}(\mu,\nu,\lambda)$ and 
its derived $t$-arrays $T_1(H)=(x^{(i)}_j)$ and $T_2(H)=(y^{(i)}_j)$,
$T^*_1(H)$ satisfies $\varepsilon^{(i)}_j(T_1^*(H)) \leq \nu^*_i - \nu^*_{i+1}$ 
if and only if $T_2(H)$ satisfies IC(2). 
\end{proposition}
\begin{proof}

Let us assume $j>1$. Write the exponents of $T_1^{*}(H)=(s_j^{(i)})$ using
$s_{j}^{(i)} = - x_{i+1 - j}^{(i)}$.
\begin{align*}
\varepsilon _{j}^{(i)}(T_{1}^{*}(H))  =& \sum_{1\leq h<j}
\left( -x_{i-h+2}^{(i+1)} + 2x_{i-h +1}^{(i)} -x_{i-h}^{(i-1)} \right)\\
&+\left( -x_{i-j+2}^{(i+1)} + x_{i-j+1}^{(i)} \right) \\
=& \sum_{1\leq h<j}
\left( (x_{i-h +1}^{(i)}-x_{i-h+2}^{(i+1)}) -  (x_{i-h}^{(i-1)} - x_{i-h +1}^{(i)}) \right) \\
&+ \left( x_{i-j+1}^{(i)}-x_{i-j+2}^{(i+1)} \right) 
\end{align*}%
Then, using the identity \eqref{tete}, we can rewrite the exponents 
in terms of the entries of $T_2(H)$ as
\begin{align*}
\varepsilon _{j}^{(i)}(T_{1}^{*}(H))  =& \sum_{1\leq h<j}
\left( (y_{n-i}^{(n-h+1)}-y_{n-i}^{(n-h)} ) -  (y_{n-i+1}^{(n-h+1)} - y_{n-i+1}^{(n-h)}) \right) \\ 
& + \left( y_{n-i}^{(n-j+1)}-y_{n-i}^{(n-j)} \right) \\
 \leq & \sum_{1\leq h<j}
\left( (y_{n-i}^{(n-h+1)}-y_{n-i}^{(n-h)}) -  (y_{n-i+1}^{(n-h+1)} - y_{n-i+1}^{(n-h)}) \right) \\
& + \left( y_{n-i}^{(n-j+1)}-y_{n-i+1}^{(n-j+1)} \right) 
\end{align*}
where the inequality is by IC(2): $y_{n-i}^{(n-j)} \geq y_{n-i+1}^{(n-j+1)}$ in $T_2(H)$. 
Parts of the consecutive terms in the right hand side cancel to give
\begin{align*}
\varepsilon _{j}^{(i)}(T_{1}^{*}(H))  \leq & \left( (y_{n-i}^{(n)}-y_{n-i}^{(n-j+1)} ) - ( y_{n-i+1}^{(n)} - y_{n-i+1}^{(n-j+1)}) \right) \\
& + \left( y_{n-i}^{(n-j+1)}-y_{n-i+1}^{(n-j+1)} \right) \\
=&  \left( y_{n-i}^{(n)} -  y_{n-i+1}^{(n)} \right) = \nu_{n-i} - \nu_{n-i+1}  
= \nu^{*}_{i} - \nu^{*}_{i+1}.
\end{align*}%

So the interlacing condition IC(2) for $T_2(H)$ is equivalent to 
\begin{equation*}
\varepsilon _{j}^{(i)}(T_{1}^{*}(H)) \leq \nu^{*}_{i} - \nu^{*}_{i+1}
\end{equation*} 
as required. The case $j=1$ can be shown similarly for all $i$.
\end{proof}

\subsection{}

Suppose we have a hive H. From Lemma \ref{l-HT12-w}, the weights of $T^*_1(H)$ and $T_2(H)$ are 
$\lambda^* - \nu^*$ and $\lambda-\mu$, respectively.
Theorem \ref{hive2gt2} states that $H$ is a hive if and only if $T_1(H)$ and $T_2(H)$, 
and hence $T^*_1(H)$ and $T_2(H)$, satisfy both $IC(1)$ and $IC(2)$.
Therefore since $H$ is a hive, Proposition \ref{p-HT12-1} and
Proposition \ref{p-HT12-3} imply $T^*_1(H)$ and $T_2(H)$ satisfy the exponent conditions, 
and consequently they are GZ schemes in $GZ(\mu^*, \lambda^* - \nu^*, \nu^*)$
and $GZ(\nu, \lambda - \mu, \mu)$, respectively.

Conversely, if $T^*_1(H)$ is a GZ scheme from $GZ(\mu^*, \lambda^* - \nu^*, \nu^*)$ 
it satisfies IC(1), IC(2), and the exponent condition, 
thus from Propositions \ref{p-HT12-1} -- \ref{p-HT12-3}, 
$T_2(H)$ is a GZ scheme. In particular, $T_1(H)$ and $T_2(H)$ are GT patterns,
meaning $H$ is a hive by Theorem \ref{hive2gt2}. Similarly, 
if $T_2(H) \in GZ(\nu, \lambda - \mu, \mu)$, then $H$ is a hive. 
This proves Theorem \ref{HT12}.


\section{LR Tableaux and GT Patterns I}


In this section we introduce a bijection  between LR tableaux and GZ
schemes (Theorem \ref{LRGZ1}). In proving this we will see that 
the semistandard and Yamanouchi conditions for tableaux are equivalent to,
respectively, the interlacing and exponent conditions for $t$-arrays. 

As an interesting consequence we then combine Theorem \ref{LRGZ1} with 
Theorem \ref{HT12} (1) to arrive at a correspondence between LR tableaux and 
hives (Corollary \ref{bij}).
It turns out that Corollary \ref{bij} is equivalent to \cite[(3.3)]{KTT1}, so we compare the two constructions. 
The main difference is that our method has an intermediate GZ scheme, which is an artefact of composing Theorem \ref{LRGZ1} and Theorem \ref{HT12}. 
To conclude the section we summarise how the conditions on LR tableaux (semistandard and Yamanouchi conditions), GZ schemes (semistandard and exponent conditions) and hives (the rhombus conditions) are translated by the bijections. 


The reader may find relevant results and further developments in, for example, 
\cite{BK, BZ1, BZ2, BZ3, DK05, KTT1, KTT2, PV05}.

\subsection{A well-known bijection between semistandard tableaux and GT patterns}

Our bijection between LR tableaux and GZ schemes is an extension of a well-known bijection
between semistandard tableaux and GT patterns, seen in, for example, \cite{GZ85}. We now review this bijection 
and state it in the form most useful for our purposes. For this we require some relevant notation. 

A non-skew semistandard tableau $Y$ is uniquely determined by its associated matrix 
$(a_{i,j}(Y))$ where
\begin{equation}\label{aij}
a_{i,j}(Y) = \hbox{\ the number of $i$'s in the $j$th row}
\end{equation}
for all $1 \leq i, j \leq n$. Note that $a_{i,j}(Y) = 0$ for $i<j$.
We also note that $\sum_{k=1}^n a_{k,j}(Y)$ for $1 \leq j \leq n$ 
give the shape of the tableau $Y$, and $\sum_{k=1}^n a_{i,k}(Y)$ for $1 \leq i \leq n$ give 
the content of $Y$. The reader is warned that these are \textit{not} the same $a_{ij}$ as those in \cite[(3.4)]{KTT1}. 
Those label hive entries, not content of a tableau. 

We remark that if $Y$ is a semistandard tableau on the skew shape $\lambda/\mu$, then 
the $a_{i,j}(Y)$'s are well defined, and the $a_{i,j}(Y)$'s with $\lambda$ or $\mu$ uniquely define $Y$.
It is possible to develop the theory of tableaux exclusively in terms of 
their associated matrices. See \cite{DK05} for this direction.

Now consider a semistandard Young tableau. Removing all instances of the largest entry simultaneously yields a tableau with a new shape. Repeating this process, we would achieve a list of successively shrinking shapes, which written downwards would form the rows of a GT pattern. This process is a bijection. See Example \ref{GT-SSYT-bij}. 

It is easy to symbolically describe the inverse of the bijection. Given a GT pattern $T=(t_{j}^{(i)})$ of type $\lambda$ with non-negative entries, it creates a semistandard Young tableau $Y_{T}$  of shape $\lambda$ whose entries are elements of $\{1,2,\dotsc ,n\}$ and defined by
\begin{equation}\label{aij-tij}
a_{i,j}(Y_{T}) = t_{j}^{(i)}-t_{j}^{(i-1)} 
\end{equation}%
for $1 \leq i, j \leq n$ with the conventions 
\begin{equation*}
 t_{j}^{(i)} = 0 \hbox{\ \ for \ \ } j > i \geq 0.
\end{equation*}
Manipulating \eqref{aij-tij}, it follows that the bijection takes a semistandard tableau $Y$ and creates a GT pattern $T_Y=(t_{j}^{(i)})$ 
according to the rule
\begin{equation}\label{gt2ssoriginal}
t_j^{(i)} =  \sum_{k=1}^{i} a_{k,j}(Y)
\end{equation}
for $ 1 \leq j \leq i \leq n$. Since $a_{k,j}(Y) = 0$ for $k<j$ in every 
non-skew semistandard tableau $Y$, we can in fact write this as
\begin{equation}\label{gt2ss}
t_j^{(i)} =  \sum_{k=j}^{i} a_{k,j}(Y).
\end{equation}
See also, for example, \cite[\S 8.1.2]{GW09} or \cite{Ki08} for further background on this bijection. 

\begin{example} \label{GT-SSYT-bij}As an example we apply the bijection to the tableau	
\begin{equation*}
\young(112,23,3)
\end{equation*}
and list the successive shapes $\lambda^{(i)}$ as they are created. 
\begin{equation*}
\includegraphics[scale=0.8]{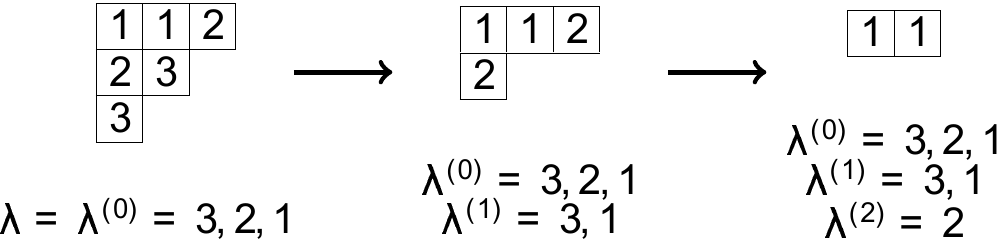}
\end{equation*}

Clearly, the shapes form a GT pattern. It is straightforward to check that the expressions \eqref{gt2ss} and \eqref{aij-tij} both hold. 
\end{example}

Under this bijection, the content of the tableau is equal to the weight of the $t$-array. 
We also note that in this bijection, the
semistandard condition on the tableau is implied by the interlacing
conditions on the $t$-array and vice versa (cf. Remark \ref{000}).

\subsection{A well-known bijection between semistandard skew tableaux and truncated GT patterns}

The bijection of $\S 4.1$ can be extended to act on skew tableaux. Again, this is a well known result included in \cite{GZ85}, \cite{GZ86} and \cite{BZ2}, among others. 

\begin{lemma}\label{conversion}
There is a bijection between the set of skew semistandard Young tableaux of 
shape $\lambda/\mu$ with entries from $\{1,2,\dotsc, n\}$ and the set of 
GT patterns for $GL_{2n}$
whose type is $\lambda'=(\lambda_{1},\dotsc ,\lambda_{n},0,\cdots ,0) \in \mathbb{Z}^{2n}$ 
and whose $k$th row is $(\mu_{1},\mu_{2},\dotsc ,\mu_{k})$ for $1\leq k\leq n$.
\begin{proof}
For a given semistandard Young tableau $Y$ of shape $\lambda/\mu$, replace the
$i$ entries with $(n+i)$'s for $1\leq i\leq n$, then fill in the empty boxes in the 
$\ell$th row of $Y$ with $\ell$'s for $1\leq \ell \leq n$. Then this process uniquely
determines a non-skew semistandard Young tableau of shape $\lambda$ with entries from 
$\{1,2,\dotsc ,2n\}$, and under the bijection given by \eqref{aij-tij},
its corresponding GT pattern for $GL_{2n}$ is the one described in the
statement.
\end{proof}
\end{lemma}

The first half of Example \ref{ex-LR-GZ} shows Lemma \ref{conversion} applied to a skew tableau. 
We remark that the GT pattern for $GL_n$ whose $k$th row is 
$(\mu_{1},\mu_{2},\dotsc ,\mu_{k})$ for $1\leq k\leq n$ corresponds to the highest
weight vector of the representation $V_n^{\mu}$ labelled by a Young diagram $\mu$. In fact, 
the GT patterns described in Lemma \ref{conversion} encode the weight vectors of 
$V_{2n}^{\lambda'}$, which are the highest weight vector for $V_{n}^{\mu}$ 
under the branching of $GL_{2n}$ down to $GL_n$. 

The bottom $n-1$ rows of a GT pattern described by Lemma \ref{conversion} hold redundant information because they are determined by $\mu$. 
It is therefore convention to omit them and achieve what is called a \textit{truncated} GT pattern.
It is also common to omit the upper-right portion of this pattern, since the interlacing conditions force those entries to be zero. For example, the first half of the bijection described by \cite[(3.3)]{KTT1} uses Lemma \ref{conversion} with these conventions.

There is an excellent example of Lemma \ref{conversion} and further explanation in \cite[\S 2]{BK}. 

\subsection{Symbolic forms of the semistandard and Yamanouchi conditions}

We are almost ready to use Lemma \ref{conversion} to establish the bijection between LR tableaux and GZ schemes. However, we first need symbolic forms of both the semistandard and Yamanouchi conditions. 

Let us express the semistandard condition for a tableau $Y$ in terms of the
$a_{i,j}(Y)$ defined in \eqref{aij}. By rearranging the entries 
in each row if necessary, we can always make the entries of $Y$ weakly increasing along each row 
from left to right. The strictly increasing condition on the columns of $Y$ can then be rephrased 
as follows: the number of entries up to $\ell$ in the $(m +1)$th row is 
not bigger than the number of entries up to $(\ell -1)$ in the $m$th row, i.e.,
\begin{equation}\label{ss-aij}
\sum_{k=1}^{\ell -1} a_{k,m}(Y) \geq \sum_{k=1}^{\ell} a_{k,m+1}(Y)
\end{equation}
for $1 \leq \ell \leq n$ and $1 \leq m < n$. 
Here, if $\ell = 1$, then the left hand side is $0$ as an empty sum 
and the inequality implies that $a_{1,m+1}(Y)=0$ for $m \geq 1$. 
Inductively, we can obtain $a_{i,m+1}(Y)=0$ for $m \geq i$ from the inequality with $\ell =i$.
This shows that for a semistandard Young tableau Y,  $a_{i,j}(Y)=0$ for $j>i$, as we noted after \eqref{aij}.

\begin{remark}\label{000}
By using the conversion formula \eqref{gt2ssoriginal}, one can directly compute that
IC(2) on a GT pattern $T$ is equivalent to the semistandard condition \eqref{ss-aij} 
in $Y_T$ corresponding to $T$.  On the other hand, IC(1) in $T$ is equivalent to 
a rather trivial condition $a_{i,j}(Y_T) \geq 0$ for all $i,j$.
\end{remark}

If $Y$ is a skew tableau of shape $\lambda/\mu$, then, using the same argument 
as for \eqref{ss-aij}, it is straightforward to see that we can make $Y$ semistandard
by rearranging elements along each row if and only if
\begin{equation}\label{skew-ss-aij}
\mu_{m+1} + \sum_{k=1}^{\ell} a_{k,m+1}(Y) \leq \mu_m + \sum_{k=1}^{\ell -1} a_{k,m}(Y)
\end{equation}
for $1 \leq \ell \leq n$ and $1 \leq m < n$.
The Yamanouchi condition in a LR tableau $Y$ can be expressed as
\begin{equation}\label{yam-aij}
\sum_{k=1}^{j} a_{i+1,k}(Y) \leq \sum_{k=1}^{j-1} a_{i,k}(Y)
\end{equation}
for $1 \leq j \leq n$ and $1 \leq i < n$.
Here, if $j=1$, then the right hand side is $0$ as an empty sum and 
the inequality implies that $a_{i+1,1}(Y)=0$ for $i \geq 1$. 
Inductively, we can obtain $a_{i+1,\ell}(Y)=0$ for $i \geq \ell$ 
from the inequality with $j = \ell$. This shows that for an LR tableau Y, $a_{i,j}(Y)=0$ 
for $i>j$, as we noted in Remark \ref{remark-Tx} (2).

\subsection{Bijection between LR tableaux and GZ schemes}
We now establish a bijection between LR tableaux and GZ schemes using Lemma \ref{conversion}. 
After applying the lemma to an LR tableau, a center section of the resulting GT pattern is removed. 
Taking the dual  of the removed array we get the desired GZ scheme. 
In doing this we observe how the conditions on the tableau become those of the scheme.

For a specific example of this bijection, see the first half of Example \ref{ex-LR-GZ}. 


\begin{theorem}\label{LRGZ1}
There is a bijection $\phi$ between $LR(\lambda/\mu,\nu)$ and 
$GZ(\mu^{\ast},\lambda^{\ast }-\nu^{\ast },\nu^{\ast })$.
In particular, the semistandard and Yamanouchi conditions in $L \in LR(\lambda/\mu,\nu)$
are equivalent to, respectively, the interlacing and exponent conditions in $\phi(L) \in 
GZ(\mu^{\ast},\lambda^{\ast }-\nu^{\ast },\nu^{\ast })$.
\end{theorem}
\begin{proof}
Let $L\in LR(\lambda/\mu,\nu)$ be given. By applying Lemma \ref{conversion} we find its corresponding GT pattern $T=(t_{j}^{(i)})$ for $GL_{2n}$ and remove the bottom $n-1$ rows to achieve a truncated GT pattern of $n+1$ rows. 
Furthermore, the truncated pattern for $L$ can be divided into three
subtriangular arrays $T_{X},T_{Y}$ and $T_{Z}$, as in Figure 2. Note that these are the same size. 
 \begin{figure}[hb]\label{fig1}
 \centering
 \includegraphics[scale=0.8]{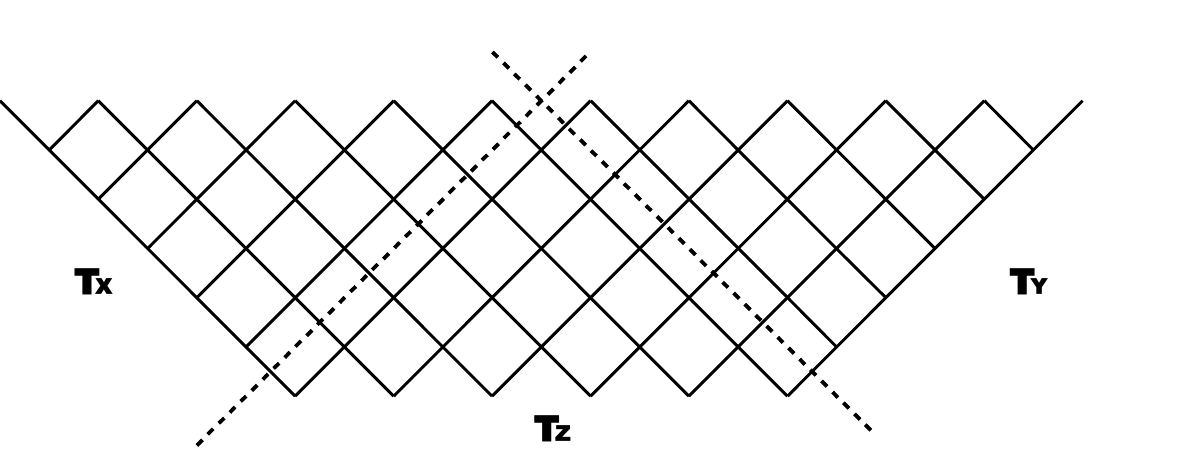}
 \caption{Dividing a truncated GT pattern into 3 subpatterns.}
 \end{figure}

The upper left subarray $T_{X}$ is completely determined by $\lambda$
because of the Yamanouchi condition (see Remark \ref{remark-Tx} (2)). The upper right subarray $T_{Y}$ contains only zeroes. Therefore, given fixed $\lambda, \mu$, 
and $\nu$, the LR tableau $L\in LR(\lambda/\mu,\nu)$ is uniquely determined by $T_{Z}$. 
We want to show that the dual array $T_{Z}^{\ast}$ of $T_{Z}$ is an element 
of $GZ(\mu^{\ast },\lambda^{\ast }-\nu^{\ast },\nu^{\ast })$, and from that establish a bijection
\begin{eqnarray*}
LR(\lambda/\mu,\nu) & \longrightarrow & 
 GZ(\mu^{\ast },\lambda^{\ast }-\nu^{\ast },\nu^{\ast })  \\
L & \longmapsto & T_{Z}^{\ast}  \hspace{35mm}^{.}
\end{eqnarray*}

Let us rewrite the middle subarray $T_Z$ as follows by reflecting it 
over a horizontal line.
\begin{equation*}
T_{Z}=%
\begin{array}{ccccccccc}
t^{(n)}_{1} &  & t^{(n)}_{2} &  & \cdots &  & t^{(n)}_{n-1} &  & t^{(n)}_{n} \\
& t^{(n+1)}_{2} &  & t^{(n+1)}_{3} &  & \cdots &  & t^{(n+1)}_{n} &  \\
&  & t^{(n+2)}_{3} &  &  &  & t^{(n+2)}_{n} &  &  \\
&  &  & \ddots &  & \Ddots &  &  &  \\
&  &  &  & t^{(2n-1)}_{n} &  &  &  &
\end{array}%
\end{equation*}%
Then $\mu_i=t^{(n)}_{i}$ for $1 \leq i \leq n$ and $T_{Z}$ satisfies the
interlacing conditions induced from the truncated GT pattern $T$, which are assured 
by the semistandardness of $L$. Therefore $T_{Z}$ is a GT pattern
of type $\mu$. From the fact that the weights of $T_{X}$, $T_{Y}$, and $T$ are
$(\lambda_{1},\dotsc ,\lambda_{n})$, $(0,\dotsc ,0)$, and 
$(\mu_{1},\dotsc ,\mu_{n}, \nu_{1},\dotsc,\nu_{n} )$ respectively, it is easy to show that the
weight of $T_{Z}$ is $\nu^{\ast}-\lambda^{\ast}$. Hence its dual $T_{Z}^{\ast}$
is a GT pattern (see \S 3.4) of type $\mu^{\ast }$ and weight $\lambda^{\ast }-\nu^{\ast }$. Next, we want to show
that $T_{Z}^{\ast }$ satisfies the exponent conditions.

Let $a_{i,j}=a_{i,j}(L)$, i.e., be the number of $i$'s in 
the $j$th row of $L$ for all $i$ and $j$. Then
\begin{equation}\label{aaa}
a_{i,j} = t^{(n+i)}_{j} - t^{(n+i-1)}_{j} \hbox{\ \ and\ } 
a_{k,k} = \lambda_{k} - t^{(n+k-1)}_{k}
\end{equation}%
for $1\leq i<j\leq n$ and $1\leq k\leq n$. Since the content of $L$ is $\nu$ with
$\nu_{q}=\sum_{k=1}^{n}a_{q,k}$ for $1\leq q \leq n$, we can write
\begin{equation}
(-\nu_{i+1})-(-\nu_{i})=\sum_{k=1}^{n}(a_{i,k}-a_{i+1,k})  \label{num1}
\end{equation}
for $1 \leq i < n$.

On the other hand, from the Yamanouchi condition \eqref{yam-aij} in $L$, we have
\begin{equation*}
\sum_{k=1}^{j} a_{i+1,k}\leq \sum_{k=1}^{j-1}a_{i,k}
\hbox{\ \ or equivalently, \  }
a_{i+1,j}\leq \sum_{k=1}^{j-1}(a_{i,k}-a_{i+1,k}).
\end{equation*}%
Then, using this inequality, (\ref{num1}) becomes
\begin{equation*}
(-\nu_{i+1})-(-\nu_{i})\geq \sum_{k=j+1}^{n} \left( a_{i,k}-a_{i+1,k} \right)+a_{i,j}
\end{equation*}%
and the right hand side is, via (\ref{aaa}), the exponents of $T_{Z}^{\ast}$.
Therefore, $T_{Z}^{\ast} \in GZ(\mu^{\ast },\lambda^{\ast }-\nu^{\ast },\nu^{\ast })$.
\end{proof}

\subsection{A bijection between LR tableaux and hives} 
Composing Theorem \ref{LRGZ1} and Theorem \ref{HT12} (1) gives a bijection between 
the set of LR tableaux and the set of hives.
\begin{equation*}
\begin{array}{ccccc}
                &      &  GZ(\mu^*, \lambda^* - \nu^*,\nu^*) &       &   \\
         & \swarrow \nearrow &             & \searrow \nwarrow  &  \\
LR(\lambda/\mu,\nu)   &      &      &      &  \mathcal{H}^{\circ}(\mu,\nu,\lambda)
\end{array}
\end{equation*}

\begin{corollary}\label{bij} \cite[(3.3)]{KTT1}
There is a bijection between $LR(\lambda/\mu,\nu)$ and
$\mathcal{H}^{\circ}(\mu,\nu,\lambda)$.
\end{corollary}
\begin{proof}
For $L \in LR(\lambda/\mu,\nu)$, we compute the corresponding truncated GT pattern
and its middle subarray $T_Z$. Then, by Theorem \ref{LRGZ1}, its dual $T^*_Z$ belongs 
to $GZ(\mu^*, \lambda^* - \nu^*,\nu^*)$. Similarly, for 
$H \in \mathcal{H}^{\circ}(\mu,\nu,\lambda)$, its first derived subarray $T_1(H)$
satisfies $T^*_1(H) \in GZ(\mu^*, \lambda^* - \nu^*,\nu^*)$ by Theorem \ref{HT12}. We can therefore identify a $H$ such that $T_1(H)=T_Z$ and this gives us a bijection from
$LR(\lambda/\mu,\nu)$ to $\mathcal{H}^{\circ}(\mu,\nu,\lambda)$.
\end{proof}

We give an example of Corollary \ref{bij} below. 

\begin{example}\label{ex-LR-GZ}
We start by using Theorem \ref{LRGZ1} to map the LR tableau below to a GT pattern (whose dual array is a GZ scheme). 

Let $\lambda=(11,7,5,3), \mu=(5,3,1,0)$ and $\nu=(7,5,3,2)$. 
The LR tableau from $LR(\lambda/\mu,\nu)$
\begin{equation*}
\young(\ \ \ \ \ 111111,\ \ \ 1222,\ 2333,244)
\end{equation*}
considered as an object for $GL_4$, corresponds to the following truncated
GT pattern.%
\begin{equation*}
\begin{array}{ccccccccccccccc}
11 &  & 7 &  & 5 &  & 3 &  & 0 &  & 0 &  & 0 &  & 0 \\
& 11 &  & 7 &  & 5 &  & 1 &  & 0 &  & 0 &  & 0 &  \\
&  & 11 &  & 7 &  & 2 &  & 1 &  & 0 &  & 0 &  &  \\
&  &  & 11 &  & 4 &  & 1 &  & 0 &  & 0 &  &  &  \\
&  &  &  & 5 &  & 3 &  & 1 &  & 0 &  &  &  &
\end{array}%
\end{equation*}%
Taking out the middle section, we find $T_{Z}$ is
\begin{equation*}
\begin{array}{ccccccc}
5 &   & 3 &   & 1 &   & 0 \\
  & 4 &   & 1 &   & 0 &   \\
  &   & 2 &   & 1 &   &   \\
  &   &   & 1 &   &  &  
\end{array}%
\end{equation*}%
with a dual array $T_{Z}^{\ast}$ belonging to $GZ(\mu^{\ast},\lambda^{\ast }-\nu^{\ast },\nu^{\ast })$.

For the second half of the process, we apply the bijection between hives and GZ schemes (Theorem \ref{HT12} (1)) to $T_{Z}^{\ast}$. We know the corresponding hive $H$ will have boundaries given by $\mu=(5,3,1,0), \nu=(7,5,3,2)$ 
and $\lambda=(11,7,5,3)$ so that it appears as follows
\begin{equation*}
\begin{array}{ccccccccc}
  &   &   &   & 0 &   &    &    & \\
  &  &   & 5 &   & 11 &    &    & \\
  &   & 8 &   & p &   & 18 &    & \\
  & 9 &   & q &   & r &    & 23 &  \\
9 &   & 16 &  & 21&   & 24 &    & 26
 \end{array}%
\end{equation*}%
with some inner entries $p$, $q$ and $r$. Adding $(T_{Z}^{\ast})^{\ast} = T_{Z}$ along the NE-SW diagonals as if it were $T_{1}(H)$
we find $p=15$, $q=16$ and $r=20$. 
\end{example}

Of course, there are many known bijections between LR tableaux and hives. For example, in the appendix of \cite{Bu00} Fulton gave a bijection between LR tableaux and hives using \textit{contratableaux}. It is interesting to note that his first step is to construct partitions from  the hive that are equivalent to the derived $t$-array $T_1$. However, that approach diverges from ours once he uses the partitions to form a contratableau. 

Our Corollary \ref{bij} is simpler than most other bijections between LR tableaux and hives, such as the one by Fulton, but is in fact equivalent to \cite[(3.3)]{KTT1}. There, the authors also take an LR tableau and compute the truncated GT pattern via Lemma \ref{conversion}. They then take row sums in the pattern and separate out a bottom section, which becomes the hive. This is simply our process in reverse, since we separate a section of the truncated GT pattern in Theorem \ref{LRGZ1} by removing $T_{Z}$, and then successively add those entries to the boundary of the hive in Theorem \ref{HT12} (1). 

The key difference, however, is that here we establish GZ schemes as an intermediate object in the bijection, which is absent from the simple and elegant treatment in \cite{KTT1}. This provides background as to why their simple bijection works, and also allows us to track the conditions as they move between objects (see tables \ref{KTT1-conds} and \ref{our-conds}). We also note that Corollary \ref{bij} is not the main focus of our discussion. Rather, it is an interesting consequence that appears when piecing together two sets of combinatorial theory -- the derived $t$-arrays of hives on one hand, and the classical bijection between tableaux and  $t$-arrays on the other. 

To complete the section we combine the two approaches for some insights. Though not clear from our presentation, the elegant formula \cite[(3.4)]{KTT1} states that the elements of the hive $H = (h_{l,m})$ are given by
\begin{equation*}
h_{l,m} = 
\begin{array}{l}
\textrm{the number of empty boxes and entries $\leq l$} \\
\textrm{in the first $m$ rows of the LR tableau.}
\end{array}
\end{equation*}

Finally, in proving \cite[Proposition 3.2]{KTT1}, King \textit{et al.} show that, under their bijection, conditions\footnote{The conditions RC(1), RC(2) and RC(3) are referred to as R2, R3 and R1 respectively by \cite{KTT1}.} on hives correspond to conditions on LR tableaux. Table \ref{KTT1-conds} summarises these equivalences. 

\begin{table}[h!] 
\centering
  \begin{tabular}{ | c | c |}
    \hline
    \textbf{Hive} & \textbf{LR tableau} \\ \hline
    RC(1)  & trivial  \\ \hline
    RC(2) & Semistandard condition \\ \hline
    RC(3)  & Yamanouchi \\ \hline
  \end{tabular}
  \caption{Equivalent LR tableau and \\
  hive conditions in \cite[(3.3)]{KTT1}}\label{KTT1-conds}
\end{table}

Using our results from Proposition \ref{IC-Rhom}, Theorem \ref{HT12}, Remark \ref{000} and Theorem \ref{LRGZ1} we are able to add a middle column showing the equivalent conditions in the intermediate GZ scheme object. See Table \ref{our-conds}.
\begin{table}[h!] 
\centering
  \begin{tabular}{| c | c | c| }
    \hline
    \textbf{Hive} & \textbf{GZ scheme} &\textbf{LR tableau} \\ \hline
  RC(1) & IC(1) & trivial \\ \hline
  RC(2) & IC(2) & Semistandard condition \\ \hline    
  RC(3) & Exponents & Yamanouchi \\ \hline
  \end{tabular}
  \caption{Equivalent LR tableau, GZ scheme \\
 and hive conditions in Corollary \ref{bij}}\label{our-conds}
\end{table}


\section{LR Tableaux and GT Patterns II}


In this section, we show that the semistandard and Yamanouchi conditions 
for tableaux are equivalent to, respectively, the exponent and semistandard conditions
for their \textit{companion tableaux}. This correspondence is obtained by taking
the transpose of matrices describing tableaux. As a result, we show that the companion
tableaux of LR tableaux are GZ schemes under the tableau-pattern bijection.

\subsection{}

For a (skew) semistandard tableau $Y$, as in \eqref{aij}, we let $a_{i,j}(Y)$ denote
the number of $i$'s in the $j$th row.

\begin{definition}
For a (skew) semistandard tableau $Y$, its \textit{companion tableau} $Y^c$
is defined as a non-skew tableau whose entries are weakly increasing along each row and whose number of $i$'s in the $j$th row is equal to $a_{j,i}(Y)$; that is, for $1 \leq i,j \leq n$,
\begin{equation}\label{companion}
a_{i,j}(Y^c) = a_{j,i}(Y).
\end{equation}
\end{definition}

\begin{example}
For the LR tableau $Y$ from Example \ref{ex-LR-GZ}, the associated matrix is
\begin{equation*}
 a_{i,j}(Y) = 
\left[
\begin{array}{cccc}
  6 & 1  & 0 & 0 \\
  0 & 3  & 1 & 1 \\  
  0 & 0  & 3 & 0 \\
  0 & 0  & 0 & 2 
\end{array}%
\right].
\end{equation*}
Then, from its transpose, we have the following companion tableau $Y^c$.
\begin{equation*}
\young(1111112,22234,333,44)
\end{equation*}
Note that $Y$ is of shape $(11,7,5,3)/(5,3,1,0)$ and content $(7,5,3,2)$ while
its companion tableau $Y^c$ is of shape $(7,5,3,2)$ and content $(6,4,4,3)$, 
which is $(11,7,5,3)-(5,3,1,0)$. The GT pattern $T_{Y^c}$ corresponding to $Y^c$ is
\begin{equation*}
\begin{array}{ccccccc}
  7 &   & 5 &   & 3 &   &  2 \\
    & 7 &   & 4 &   & 3 &   \\  
    &   & 7 &   & 3 &   &   \\
    &   &   & 6 &   &  &  
\end{array}.
\end{equation*}
We want to show that this correspondence $Y \mapsto T_{Y^c}$ gives 
another bijection from the set of LR tableaux to the set of GZ schemes.
\end{example}

In \cite{vL01}, van Leeuwen replaced the Yamanouchi condition in LR tableaux
with the semistandard condition in their companion tableaux. Here, we show that 
the semistandard condition in LR tableaux has a counterpart in the companion 
tableaux as well, and then we identify the companion tableaux as an independent
object equivalent to GZ schemes.

\begin{theorem}\label{CTGZ}
For a LR tableau $Y$, we let $Y^{c}$ denote its companion tableau and 
let $T_{Y^{c}}$ denote the GT pattern corresponding to $Y^{c}$. 
The map $\psi(Y)=T_{Y^{c}}$ gives a bijection from $LR(\lambda/\mu, \nu)$ to 
$GZ(\nu, \lambda - \mu, \mu)$. In particular, the Yamanouchi and semistandard 
conditions in $Y$ are equivalent to, respectively, the interlacing condition IC(2) 
and the exponent condition in $T_{Y^{c}}$.
\end{theorem}
\begin{proof}
From \eqref{companion}, $Y$ is a tableau of shape $\lambda/\mu$
if and only if the content of $Y^c$ is equal to $\lambda - \mu$. The content of $Y$ is equal 
to the shape of $Y^c$. The type and weight of $T_{Y^{c}}$ are therefore $\nu$ and $\lambda - \mu$, respectively.

Recall the Yamanouchi condition in $Y$ \eqref{yam-aij}: for $0 \leq i < n$ and $1 \leq j < n$,
\begin{equation}\label{2-1}
\sum_{k=1}^{i} a_{j,k}(Y)   \geq \sum_{k=1}^{i+1} a_{j+1,k}(Y).
\end{equation}%
Since $a_{i,j}(Y)=a_{j,i}(Y^{c})$ for all $i$
and $j$, this inequality, in terms of the entries in $Y^{c}$, is saying that
the number of entries less than or equal to $i+1$ in the $(j+1)$th row is
not more than the number of entries less than or equal to $i$ in
the $j$th row. It is the semistandard condition for $Y^{c}$ and therefore the
interlacing condition for $T_{Y^{c}}$.

To show this, consider expressing the elements of the GT pattern 
$T_{Y^{c}} = (t_j^{(i)})$ in terms of  $a_{i,j}(Y^{c})$. From the standard 
bijection between semistandard tableaux and GT patterns, \eqref{gt2ssoriginal}, 
we have
\begin{equation*}
t_j^{(i)} =  \sum_{k=1}^{i} a_{k,j}(Y^{c})
\end{equation*}
where $a_{i,j}(Y^{c})$ is the number of $i$ entries in the $j$th row of $Y^{c}$.

Consider the interlacing condition IC(2): $t_j^{(i)} \geq t_{j+1}^{(i+1)}$ where $0 \leq i < n$ and $1 \leq j < n$. 
Writing this with the above relation gives
\begin{equation*}
\sum_{k=1}^{i} a_{k,j}(Y^{c}) \geq \sum_{k=1}^{i+1} a_{k,j+1}(Y^{c}) 
\ \ \Leftrightarrow \ \ \sum_{k=1}^{i} a_{j,k}(Y)  \geq  \sum_{k=1}^{i+1} a_{j+1, k}(Y)
\end{equation*}
which is exactly the expression for the Yamanouchi condition \eqref{2-1}. 
It can be similarly shown that, as mentioned in Remark \ref{000}, IC(1) is
equivalent to $a_{i,j}(Y) \geq 0$.

\smallskip

Using \eqref{skew-ss-aij}, the semistandard condition for $Y$ says we have, for all $1 \leq \ell \leq n$ and $1 \leq m < n$,
\begin{equation}
\left( \sum_{k=1}^{\ell} a_{k,m+1}(Y) - \sum_{k=1}^{\ell -1} a_{k,m}(Y) \right)
 \leq \left( \mu_m - \mu_{m+1} \right)
\end{equation}
or
\begin{equation*}
 \sum_{k=1}^{\ell-1} \left( a_{m+1,k}(Y^{c}) - a_{m,k}(Y^{c}) \right) + 
a_{m+1,\ell}(Y^{c}) \leq \left( \mu_m - \mu_{m+1} \right).
\end{equation*}

To finish our proof, it is enough to show that the left hand side of the above inequality 
is the exponent $\varepsilon^{(m)}_{\ell}(T_{Y^{c}})$. This can be easily seen, by using \eqref{gt2ss}, as
\begin{align*}
\varepsilon _{\ell}^{(m)}(T_{Y^{c}})  =& \sum_{1\leq h<\ell} \left( t_{h}^{(m+1)}-2t_{h}^{(m)}+t_{h}^{(m-1)} \right)
+ \left( t_{\ell}^{(m+1)}-t_{\ell}^{(m)} \right) \\
=& \sum_{1\leq h<\ell} \left( \sum_{k=h}^{m+1} a_{k,h}(Y^{c}) -2 \sum_{k=h}^{m} a_{k,h}(Y^{c})
+ \sum_{k=h}^{m-1} a_{k,h}(Y^{c}) \right) \\
 &+ \left( \sum_{k=\ell}^{m+1} a_{k,\ell}(Y^{c}) -\sum_{k=\ell}^{m} a_{k,\ell}(Y^{c}) \right) \\
=& \sum_{1\leq k< \ell} \left( a_{m+1,k}(Y^{c}) -  a_{m,k}(Y^{c}) \right) +  a_{m+1,\ell}(Y^{c}).
\end{align*}
\end{proof}

We now have an alternative proof of Corollary \ref{bij}.

\begin{corollary}
There is a bijection between $LR(\lambda /  \mu, \nu)$ 
and $\mathcal{H}^{\circ}(\mu, \nu, \lambda)$. 
\end{corollary}
\begin{proof}
We can map any $Y \in LR(\lambda /  \mu, \nu)$ to 
$T_{Y^{c}} \in GZ(\nu, \lambda - \mu, \mu)$ via the bijection in 
Theorem \ref{CTGZ}. From Theorem \ref{HT12} there is a bijection 
between $\mathcal{H}^{\circ}(\mu, \nu, \lambda)$ and $GZ(\nu, \lambda - \mu, \mu)$ 
through the derived t-array $T_2$ of a hive. The composition of 
the first bijection with the inverse of the second then gives a bijection 
which assigns $Y \in LR(\lambda /  \mu, \nu)$ to 
$H \in \mathcal{H}^{\circ}(\mu, \nu, \lambda)$ if and only if $T_2 (H) = T_{Y^{c}}$. 
\end{proof}

\medskip

\subsection*{Acknowledgement}
We thank Roger Howe and Victor Protsak 
for helpful conversations regarding several aspects of this work.

\end{document}